\documentclass{amsart}

\usepackage{latexsym}
\usepackage{enumerate}
\usepackage{amssymb}
\usepackage{amsfonts}
\usepackage{amsmath}
\usepackage{mathrsfs}
\usepackage{amsthm}
\usepackage{geometry}
\usepackage[colorlinks,
linkcolor=red,
anchorcolor=blue,
citecolor=green
]{hyperref}
\usepackage{varioref}
\usepackage{hyperref}
\usepackage{cleveref}
\usepackage[all]{xy}

\newcommand{\Ric}{\mathrm{Ric}}

\newcommand{\Vol}{\mathrm{Vol}}

\newcommand{\dist}{\mathrm{dist}}

\newtheorem{thm}{Theorem}[section]

\newtheorem{fact}{Fact}[section]
\newtheorem{prop}{Proposition}[section]

\newtheorem{lmm}{Lemma}[section]
\newtheorem{question}{Question}[section]

\newtheorem*{rmk}{Remark}

\theoremstyle{remark}

\title{ On the Gromov-Hausdorff limits of Tori with Ricci conditions }

\author{Shengxuan Zhou}
\address{Beijing International Center for Mathematical Research\\
Peking University\\
Beijing\\ 100871\\ China}
\address{Current address: Institut de Math\'ematiques de Toulouse\\
Universit\'e Toulouse III - Paul Sabatier\\
Toulouse\\ 31062\\ France}

\email{zhoushx98@outlook.com}

\thanks{}
\date{}
\dedicatory{}

\begin{document}

\begin{abstract}
Let $n\geq 4$. In this paper, we construct a sequence of smooth Riemannian metrics $g_i $ on $\mathbb{R}^n$ such that:
\begin{enumerate}[1)]
\item $g_i = g_{\rm Euc} $ outside the standard Euclidean unit ball $B_1 (0) \subset \mathbb{R}^n $,
\item $\Ric_{g_i} \geq -\Lambda $ and $ {\rm diam} \left( B_1 (0) ,g_i \right) \leq D $ for some $\Lambda ,D>0$ independent of $i$,
\item The pointed Gromov-Hausdorff limit of $(\mathbb{R}^n ,g_i) $ is a topological orbifold but not a topological manifold.
\end{enumerate}

As a consequence, for $n\geq 4$, we can find a sequence of tori $(T^n , g_i )$ with Ricci lower bound and diameter bound such that the Gromov-Hausdorff limit is not a topological manifold. This answers a question of Bru\`e-Naber-Semola \cite{bruenabersemola1} in the negative. In $4$-dimensional case, we prove that the Gromov-Hausdorff limit of tori with $2$-side Ricci bound and diameter bound is always a topological torus. In the K\"ahler case, the Gromov-Hausdorff limit of K\"ahler tori of real dimension $4$ with Ricci lower bound is always a topological orbifold with isolated singularities, and the only type of singularities is $\mathbb{R}^4 / Q_8 $.
\end{abstract}
\maketitle
\tableofcontents

\section{Introduction}

In \cite{bruenabersemola1}, Bru\`e-Naber-Semola proved that for a sequence of Riemannian manifolds with bounded diameter and sectional curvature lower bound, if they are all homeomorphic to a torus, then its Gromov-Hausdorff limit is a topological torus, which answers a question of Petrunin \cite{Zamora2} in the affirmative. They then pose the following open question about sequences of torus satisfying a lower bound on the Ricci curvature. Note that Petrunin also posed this question \cite{Zamora3}.

\begin{question}[\cite{bruenabersemola1, Zamora3}]\label{question}
Let $(M^n_i, g_i) \xrightarrow{GH} (X^k,d_X)$ satisfy ${\rm Ric}_{g_i} \ge -(n-1)$ and $ {\rm diam} (M_i ,g_i ) \le D $, where $D>0$ is a constant independent of $i$, and $k$ is the rectifiable dimension of $X$. Assume that $M^n_i$ are all homeomorphic to tori $T^n$. Is $X^k$ homeomorphic to $T^k$?	
\end{question}

Based on the topological stability of non-collapsed sequence of three-dimensional manifolds \cite{Simon}, Bru\`e-Naber-Semola gives an affirmative answer to the above question when $n = 3$.

But in general, the answer to Question \ref{question} is negative when $k\geq 4$. The following result is our main theorem in this paper.

\begin{thm}\label{thmexamplesequencetorus}
Let $n\geq k\geq 4$. Then we can find a sequence of Riemannian metrics $g_i$ on $T^n $ satisfy ${\rm Ric}_{g_i} \ge -(n-1)$ and $ {\rm diam} (M_i , g_i) \le D $ for some $D>0$, but the Gromov-Hausdorff limit space $(X^k,d_X)$ is not a topological manifold, where $k$ is the rectifiable dimension of $X$.
\end{thm}

\begin{rmk}
In fact, in our construction, $X^k$ is a topological orbifold, and the tangent cones at any singular point is isomorphic to $C(\mathbb{S}^3 /Q_8 ) \times \mathbb{R}^{k-4} $, where $Q_8 $ is the quaternion group.
\end{rmk}

\begin{rmk}
In Theorem \ref{thmexamplesequencetorus}, our construction is actually a non-collapsed example. Here, we discuss potential obstructions to the construction of collapsed examples. By Gromov’s non-collapsed theorem for aspherical manifolds, the universal cover of any sequence of tori satisfying a Ricci curvature lower bound and a diameter upper bound must be locally non-collapsed. In this case, the convergence can always be described in terms of group actions on the non-collapsed universal coverings and their corresponding quotients. If the limit space has singularities, the group actions may lead to a large topological singular subset in the limit of the universal coverings. 

For example, in the case where $(T^4_i, g_i) \xrightarrow{GH} (X^k,d_X)$ with $k\leq 3$, if such an example exists, then on the $4$-dimensional non-collapsed Ricci limit space given by the sequence of universal coverings, a topological singular set with positive Hausdorff dimension may occur. This seems to contradict certain known results, such as the recent strong characterization given by Brue-Pigati-Semola \cite{bps1} of the topology of $4$-dimensional non-collapsed tangent cones. Hence, this could serve as a potential obstruction to constructing such examples in this case.
\end{rmk}

We now give an outline of the proof of Theorem \ref{thmexamplesequencetorus}. Obviously, we only need to prove it for the case $n=k=4$, and the examples in general case can be obtained by crossing a flat $T^{k-4}$ and a sequence of small flat $T^{n-k}$. Note that $\mathbb{R}P^2 $ can be viewed as an embedded submanifold in $T^4 $ and that the boundary of a tubular neighborhood of $\mathbb{R}P^2 $ is diffeomorphic to $\mathbb{S}^3 /Q_8$. Then we find a sequence of smooth Riemannian metrics on the tubular neighborhood of $\mathbb{R}P^2 $, which are equal near the boundary of the tubular neighborhood, but the restrictions on $\mathbb{R}P^2 $ converge to $0$. Our next step is to construct a sequence of metrics on $T^4 $, whose Gromov-Hausdorff limit is homeomorphic to $T^4 /\thicksim $, where $\thicksim $ is the equivalence relation: $x\thicksim y $ if and only if $x,y\in \mathbb{R}P^2 $. Finally, the boundary of a tubular neighborhood of $\mathbb{R}P^2 $ is diffeomorphic to $\mathbb{S}^3 /Q_8$ showing that there exists an open neighborhood of $[x]$ homeomorphic to $C(\mathbb{S}^3 /Q_8 ) $, where $x\in \mathbb{R} P^2$. It follows that $T^4 /\thicksim $ is not a topological manifold.

The key step in the proof of the above theorem is to construct a sequence of Riemannian metrics $\mathbb{R}^n$ as follows. 

\begin{prop}\label{propRnmetric}
    Let $n\geq 4$. Then there exists a sequence of smooth Riemannian metrics $g_i $ on $\mathbb{R}^n$ such that:
\begin{enumerate}[1)]
\item $g_i = g_{\rm Euc} $ outside the Euclidean unit ball $B_1 (0) \subset \mathbb{R}^n $,
\item $\Ric_{g_i} \geq -\Lambda $ and $ {\rm diam} \left( B_1 (0) ,g_i \right) \leq D $ for some $\Lambda ,D>0$ independent of $i$,
\item The pointed Gromov-Hausdorff limit of $(\mathbb{R}^n ,g_i) $ is a topological orbifold but not a topological manifold.
\end{enumerate}
\end{prop}

Although the Gromov-Hausdorff limit of tori with Ricci curvature lower bound is not necessary to homeomorphic to a torus when $n\geq 4$, it seems that we may still have some restrictions on the topology of the limit space. For example, in the $4$-dimensional case, we can show that the Gromov-Hausdorff limit of tori with $2$-side Ricci bound is always a topological torus:

\begin{prop}\label{propT42-sidericcibound}
Let $(M^4_i, g_i) \xrightarrow{GH} (X^k,d_X)$ satisfy $\left| {\rm Ric}_{g_i} \right| \le 3 $ and $ {\rm diam} (M_i ,g_i ) \le D $, where $D>0$ is a constant independent of $i$, and $k$ is the rectifiable dimension of $X$.  Assume that $M^4_i$ are all homeomorphic to tori $T^4$. Then $X^k$ is homeomorphic to $T^k$. Moreover, if $k\leq 3$, then $M_i$ is diffeomorphic to the standard $T^4$ for sufficiently large $i$.
\end{prop}

The proof of this proposition is to rule out the above singularities by using topological conditions when $\dim M_i =4$ and $\left| {\rm Ric}_{g_i} \right| \le 3 $. By the same argument, one can show that the limit of a non-collapsing sequence of $4$-dimensional spheres with uniform $2$-side Ricci bound is also a $4$-dimensional sphere. See Lemma \ref{lmmT4noncollap2-sidericcibound}.

From Proposition \ref{propT42-sidericcibound}, we see that a $4$-dimensional torus is diffeomorphic to the standard $T^4$ if and only if the metric sequence with bounded Ricci curvature and bounded diameter can collapse. But to the best of the author's knowledge, it is not known whether the $4$-dimensional exotic torus exists.

In the K\"ahler case, the Gromov-Hausdorff limit of $4$-dimensional tori with Ricci lower bound is always a topological orbifold with isolated singularities, and the only type of singularities is $\mathbb{R}^4 / Q_8 $.

\begin{prop}\label{propT4kahler}
Let $(M_i^4 , g_i) \xrightarrow{GH} (X^k,d_X)$ satisfy ${\rm Ric}_{g_i} \ge -3$ and $ {\rm diam} (M_i , g_i) \le D $, where $D>0$ is a constant independent of $i$, and $k\leq 4$ is the rectifiable dimension of $X$. Assume that $g_i$ are K\"ahler and $M^4_i$ are all homeomorphic to tori $T^4$. Then $X $ is homeomorphic to $T^k$ when $k\leq 3$, and $X $ is a topological orbifold when $k=4$. Moreover, when $k=4$, every point of $X$ has a neighborhood that is homeomorphic to an open subset of $ \mathbb{R}^4 /Q_8 $.
\end{prop}

This paper is organized as follows. In Section \ref{sectionconesmoothing}, we show that the metric cone $C(\mathbb{S}^3 /Q_8 ) $ is smoothable in $\mathbb{R}^4$. The proofs of Theorem \ref{thmexamplesequencetorus} and Proposition \ref{propRnmetric} are presented in Section \ref{sectionsequencetorus}. The proofs of Proposition \ref{propT42-sidericcibound} and \ref{propT4kahler} are contained in Section \ref{4dimcase}. In Appendix \ref{ghpolarized2dim}, we give a refinement of Liu-Sz\'ekelyhidi's results in \cite{lggs1, lggs2} on K\"ahler surfaces.

\vspace{0.2cm}

\textbf{Acknowledgement.} The author wants to express his deep gratitude to Professor Gang Tian for constant encouragement. He would also be grateful to Jiyuan Han, Wenshuai Jiang, Zuohai Jiang, Aaron Naber, Daniele Semola and Fan Ye for valuable conversations. In addition, the author would like to thank Sergio Zamora for informing him that Petrunin had also posed Question \ref{question}. The author is supported by National Key R\&D Program of China 2023YFA1009900. 

\section{\texorpdfstring{$C(\mathbb{S}^3 /Q_8 )$}{Lg} is smoothable in \texorpdfstring{$\mathbb{R}^4 $}{Lg} }
\label{sectionconesmoothing}

Let $P = f( \mathbb{R}P^2 ) $, where $f$ is the standard embedding
$$ f: \mathbb{R}P^2 = S^2 /\{ \pm 1 \} \to \mathbb{R}^4,\; [x_1 ,x_2 ,x_3] \mapsto ( x_1^{2} -x_2^{2} , x_1 x_2 ,x_1 x_3 ,x_2 x_3 ) .$$ 
Then the normal bundle $N_{P}$ of $P$ is a $2$-dimensional vector bundle has Euler number $\pm 2$, and the double covering $\tilde{N}_{P}$ is a $2$-dimensional vector bundle on $S^2 $ which has Euler number $\pm 4$. Hence we see that the vector bundle $\tilde{N}_{P}$ is isomorphic to the line bundle $\mathcal{O} (-4)$ on $\mathbb{C}P^1 \cong S^2 $ in the $C^\infty $ sense. The $S^1$-bundle corresponding to $N_P$ and $\tilde{N}_P$ is diffeomorphic to $\mathbb{S}^3 /Q_8 $ and the lens space $L(4,4) $, respectively. More details can be found in \cite{lawson1}.

Now we recalling the definition of Berger sphere. Here a Berger sphere is a standard $3$-sphere with Riemannian metric from a one-parameter family, which can be obtained from the standard metric by shrinking along fibers of a Hopf fibration $\mathbb{S}^3 \to S^2 $. Let $X_1$, $X_2$, $X_3$ be the left-invariant vector fields on $SU(2)$ spanned by 
$$	\begin{bmatrix} 
	\sqrt{-1} & 0 \\
	0 & -\sqrt{-1} \\
	\end{bmatrix},
	\quad
	\begin{bmatrix} 
	0 & 1 \\
	-1 & 0 \\
	\end{bmatrix},
	\quad
	\begin{bmatrix} 
	0 & \sqrt{-1} \\
	\sqrt{-1} & 0 \\
	\end{bmatrix} , $$
respectively. Then for any $t>0$, there exists a left-invariant metric $g_t$ on $SU(2)$ such that $t^{-1} X_1 $, $X_2$, $X_3$ is an orthonormal frame. By the diffeomorphism $\mathbb{S}^3 \to SU(2) ,\; (z,w) \mapsto \begin{bmatrix} 
	z & -\bar{w} \\
	w & \bar{z} \\
	\end{bmatrix} $, the left-invariant metric gives a Riemannian metric on $\mathbb{S}^3$, and the standard Hopf fibration $\mathbb{S}^3 \to S^2 $ is a Riemannian submersion. Let $\{ \sigma_i \}_{i=1}^3 $ be the the basis of left-invariant forms dual to $\{ X_i \}_{i=1}^3 $. See also \cite[Subsection 4.4.3]{pp1}.
	
	Let $\rho ,\phi $ be smooth nonnegative functions on $[0,\infty )$ such that $\rho ,\phi $ are positive on $(0,\infty )$,
	$$ \rho (0) = 1 ,\;\rho^{\mathrm{(odd)}} (0) =0 ,$$
and
    $$ \phi (0) =0 ,\; \phi' (0)=4 ,\; (\rho \phi)^{\mathrm{(even)}} (0) =0 .$$
    Then the construction
    \begin{eqnarray}\label{constructionO-4metric}
    \bigg( (0,\infty ) \times \mathbb{S}^3 , g_{\rho ,\phi} \bigg) = \bigg( (0,\infty ) \times \mathbb{S}^3 , dr^2 + \rho^2 (r) \left[ \phi^2 (\sigma^1)^2 + (\sigma^2)^2 + (\sigma^3)^2 \right] \bigg) .
    \end{eqnarray}
    yields a singular Riemannian metric on the complex blow-up space $\hat{\mathbb{C}}^2 \cong \mathrm{Total} ( \mathcal{O} (-1) ) $ of $\mathbb{C}^2$ at the origin $0$ with conical singularities of cone angle $8\pi $ along the exceptional divisor $\mathbb{C}P^1$, where $\mathrm{Total} (\mathcal{O} (-1))$ is the total space of $\mathcal{O} (-1) $. Roughly speaking, ``conical singularities with cone angle $\theta$" here means it looks like $\mathbb{C}\times C(\mathbb{S}^1_{\theta}) $ locally. Moreover, through the ramified covering map $\mathcal{O} (-1) \to \mathcal{O} (-4) $ induced by the Hopf sub-action of $\mathbb{Z}/4\mathbb{Z}$, one can see that (\ref{constructionO-4metric}) gives a smooth Riemannian metric on the total space of $\tilde{N}_P \cong \mathcal{O} (-4)$. Since the metric on Berger sphere is invariant under both the actions generated by Hopf sub-action $\mathbb{Z}/4\mathbb{Z}$ and the action $(z_1 ,z_2) \mapsto (\bar{z}_2 , -\bar{z}_1) $, one can see that the above construction gives a smooth Riemannian metric on the total space of $N_P$. Write $X_0 =\frac{\partial }{\partial r} $. Then the Ricci curvature of $g_{\rho ,\phi}$ can be expressed as following.

\begin{lmm}\label{lmmriccicurvature}
Let $\rho$, $\phi $ and $g_{\rho ,\phi}$ be as above. Then the Ricci curvature tensor of $g_{\rho ,\phi}$ can be determined by
\begin{eqnarray}
\Ric_{g_{\rho ,\phi}} \left(X_0 \right) & = & - (3\rho^{-1} \rho'' +\phi^{-1} \phi'' +2\rho^{-1} \phi^{-1} \rho'\phi' ) X_0 ,\\
\Ric_{g_{\rho ,\phi}} \left(X_1 \right) & = & - [\rho^{-1} \rho'' +\phi^{-1} \phi'' +4\rho^{-1} \phi^{-1} \rho'\phi' - 2\rho^{-2} \phi^2 + 2\rho^{-2} (\rho')^{2} ] X_1 ,\\
\Ric_{g_{\rho ,\phi}} \left(X_2 \right) & = & [- \rho^{-1} \rho'' - \rho^{-1} \phi^{-1} \rho'\phi' +4\rho^{-2} - 2\rho^{-2} \phi^2 - 2\rho^{-2} (\rho')^{2} ] X_2 ,\\
\Ric_{g_{\rho ,\phi}} \left(X_3 \right) & = & [- \rho^{-1} \rho'' - \rho^{-1} \phi^{-1} \rho'\phi' +4\rho^{-2} - 2\rho^{-2} \phi^2 - 2\rho^{-2} (\rho')^{2} ] X_3 .
\end{eqnarray}
\end{lmm}

\begin{proof}
Set $e_0 =X_0 $, $e_1 = (\rho \phi)^{-1} X_1 $, $e_2 = \rho^{-1} X_2 $, $e_3 = \rho^{-1} X_3 $, $\omega^0 = dr $, $\omega^1 = \rho\phi\sigma^1 $, $\omega^2 =\rho \sigma^2 $, and $\omega^3 =\rho \sigma^3 $. Then $\{ e_0 ,e_1 ,e_2 ,e_3 \}$ is an orthonormal frame, and $\{ \omega^0 ,\omega^1 ,\omega^2 ,\omega^3 \}$ is the dual coframe.

Recall that $\left[ X_0 ,X_i \right] =0 $ and $\left[ X_i ,X_j \right] =2\epsilon_{ijk} X_k $, where $1\leq i,j,k\leq 3$, and $\epsilon_{ijk}$ is the Levi-Civita symbol. Then we have
\begin{eqnarray*}
\left[ e_0 ,e_1 \right] & = & \left[ X_0 ,\rho^{-1} \phi^{-1} X_1 \right] = -\left( \rho^{-1} \rho' +\phi^{-1} \phi' \right) e_1 ,\\
\left[ e_0 ,e_2 \right] & = & \left[ X_0 , \rho^{-1} X_2 \right] = - \rho^{-1} \rho' e_2 ,\\
\left[ e_0 ,e_3 \right] & = & \left[ X_0 ,\rho^{-1} X_3 \right] = - \rho^{-1} \rho' e_3 ,\\
\left[ e_1 ,e_2 \right] & = & \left[ \rho^{-1} \phi^{-1} X_1 , \rho^{-1} X_2 \right] = 2 \rho^{-1} \phi^{-1} e_3 ,\\
\left[ e_1 ,e_3 \right] & = & \left[ \rho^{-1} \phi^{-1} X_1 , \rho^{-1} X_3 \right] = -2 \rho^{-1} \phi^{-1} e_2 ,\\
\left[ e_2 ,e_3 \right] & = & \left[ \rho^{-1} X_2 ,\rho^{-1} X_3 \right] = 2 \rho^{-1} \phi e_1 .
\end{eqnarray*}

It follows that
\begin{eqnarray*}
\nabla_{e_0} e_0 & = & \nabla_{e_0} e_1 = \nabla_{e_0} e_2 =\nabla_{e_0} e_3 = 0 ,\\
\nabla_{e_1} e_0 & = & (\rho^{-1} \rho' +\phi^{-1} \phi' ) e_1 ,\\
\nabla_{e_1} e_1 & = & -(\rho^{-1} \rho' +\phi^{-1} \phi' ) e_0 ,\\
\nabla_{e_1} e_2 & = & (2\rho^{-1} \phi^{-1} - \rho^{-1} \phi ) e_3 ,\\
\nabla_{e_1} e_3 & = & (-2\rho^{-1} \phi^{-1} + \rho^{-1} \phi ) e_2 ,\\
\nabla_{e_2} e_0 & = & \rho^{-1} \rho' e_2 ,\\
\nabla_{e_2} e_1 & = & -\rho^{-1} \phi e_3 ,\\
\nabla_{e_2} e_2 & = & \nabla_{e_3} e_3 = -\rho^{-1} \rho' e_0 ,\\
\nabla_{e_2} e_3 & = & \rho^{-1} \phi e_1 ,\\
\nabla_{e_3} e_0 & = & \rho^{-1} \rho' e_3 ,\\
\nabla_{e_3} e_1 & = & \rho^{-1} \phi e_2 ,\\
\nabla_{e_3} e_2 & = & -\rho^{-1} \phi e_1 .
\end{eqnarray*}

Let $\nabla e_i = \omega_i^j e_j $. Then $\omega_i^j$ are $1$-forms, and we have Cartan's formulas $\omega_i^j =-\omega_j^i $, $d\omega^i =\omega^j \wedge \omega_j^i $, $\Omega_i^j = d\omega_i^j - \omega^k_i \wedge \omega^j_k  $ and $R_{g_{\rho ,\phi}} (X,Y)e_k = \Omega_k^j (X,Y) e_j $. Hence one can see that
\begin{eqnarray*}
\omega_0^1 & = & (\rho^{-1} \rho' +\phi^{-1} \phi' ) \omega^1 ,\\
\omega_0^2 & = & \rho^{-1} \rho' \omega^2 ,\\
\omega_0^3 & = & \rho^{-1} \rho' \omega^3 ,\\
\omega_1^2 & = & \rho^{-1} \phi \omega^3 ,\\
\omega_1^3 & = & -\rho^{-1} \phi \omega^2 ,\\
\omega_2^3 & = & (2\rho^{-1} \phi^{-1} + \rho^{-1} \phi ) \omega^1 ,
\end{eqnarray*}
\begin{eqnarray*}
d\omega^0 & = & 0 ,\\
d\omega^1 & = & (\rho^{-1} \rho' +\phi^{-1} \phi' ) \omega^0 \wedge \omega^1 - 2\rho^{-1} \phi \omega^2 \wedge \omega^3 ,\\
d\omega^2 & = & \rho^{-1} \rho' \omega^0 \wedge \omega^2 + 2\rho^{-1} \phi^{-1} \omega^1 \wedge \omega^3 ,\\
d\omega^3 & = & \rho^{-1} \rho' \omega^0 \wedge \omega^3 - 2\rho^{-1} \phi^{-1} \omega^1 \wedge \omega^2 ,
\end{eqnarray*}
and
\begin{eqnarray*}
\Omega_0^1 & = & d \omega^1_0 + \omega_i^1 \wedge \omega_0^i = \left( \rho^{-1} \rho'' +\phi^{-1} \phi'' + 2\rho^{-1}\phi^{-1} \rho' \phi' \right) \omega^0 \wedge \omega^1 - 2\rho^{-1} \phi' \omega^2 \wedge \omega^3 ,\\
\Omega_0^2 & = & d \omega^2_0 + \omega_i^2 \wedge \omega_0^i = \rho^{-1} \rho'' \omega^0 \wedge \omega^2 - \rho^{-1} \phi' \omega^1 \wedge \omega^3 ,\\
\Omega_0^3 & = & d \omega^3_0 + \omega_i^3 \wedge \omega_0^i = \rho^{-1} \rho'' \omega^0 \wedge \omega^2 + \rho^{-1} \phi' \omega^1 \wedge \omega^3 ,\\
\Omega_1^2 & = & d \omega^2_1 + \omega_i^2 \wedge \omega_1^i = - \left[ \rho^{-2} \phi^2 - \rho^{-2} (\rho'')^{2} -\rho^{-1} \phi^{-1} \rho' \phi' \right] \omega^1 \wedge \omega^2 + \rho^{-1} \phi' \omega^0 \wedge \omega^3 ,\\
\Omega_1^3 & = & d \omega^3_1 + \omega_i^3 \wedge \omega_1^i = - \left[ \rho^{-2} \phi^2 - \rho^{-2} (\rho'')^{2} -\rho^{-1} \phi^{-1} \rho' \phi' \right] \omega^1 \wedge \omega^3 - \rho^{-1} \phi' \omega^0 \wedge \omega^2 ,\\
\Omega_2^3 & = & d \omega^3_2 + \omega_i^3 \wedge \omega_2^i = -\left[ 4 \rho^{-2} -3 \rho^{-2} \phi^2 - \rho^{-2} (\rho')^{2} \right] \omega^2 \wedge \omega^3 - 2 \rho^{-1} \phi' \omega^0 \wedge \omega^1 .
\end{eqnarray*}

Since $R_{g_{\rho ,\phi}} (X,Y)e_k = \Omega_k^j (X,Y) e_j $, we have 
$$\Ric_{g_{\rho ,\phi}} (X) =\sum_k R_{g_{\rho ,\phi}}(X,e_k) e_k = \sum_{j,k} \Omega_k^j (X,e_k) e_j .$$ Then
\begin{eqnarray*}
\Ric_{g_{\rho ,\phi}} \left(X_0 \right) & = & - (3\rho^{-1} \rho'' +\phi^{-1} \phi'' +2\rho^{-1} \phi^{-1} \rho'\phi' ) X_0 ,\\
\Ric_{g_{\rho ,\phi}} \left(X_1 \right) & = & - [\rho^{-1} \rho'' +\phi^{-1} \phi'' +4\rho^{-1} \phi^{-1} \rho'\phi' - 2\rho^{-2} \phi^2 + 2\rho^{-2} (\rho')^{2} ] X_1 ,\\
\Ric_{g_{\rho ,\phi}} \left(X_2 \right) & = & [- \rho^{-1} \rho'' - \rho^{-1} \phi^{-1} \rho'\phi' +4\rho^{-2} - 2\rho^{-2} \phi^2 - 2\rho^{-2} (\rho')^{2} ] X_2 ,\\
\Ric_{g_{\rho ,\phi}} \left(X_3 \right) & = & [- \rho^{-1} \rho'' - \rho^{-1} \phi^{-1} \rho'\phi' +4\rho^{-2} - 2\rho^{-2} \phi^2 - 2\rho^{-2} (\rho')^{2} ] X_3 ,
\end{eqnarray*}
which establishes the formulas.
\end{proof}

Now we can prove the main result in this section.

\begin{prop}\label{propexamplesmoothingrhophiversion}
There exist a small constant $\varepsilon_0 >0$ and smooth nonnegative functions $\rho ,\phi $ on $[0,\infty )$ such that $\rho ,\phi $ are positive on $(0,\infty )$, $ \rho|_{[0,\varepsilon_0]} =1 $, $ \phi|_{[0,\varepsilon_0]} =4r $, $\rho' |_{[2,\infty)} =e^{-100} $, $\phi |_{[2,\infty)} =1 $, and $\Ric_{g_{\rho ,\phi} } \geq 0 $.
\end{prop}

\begin{proof}
At first, we can find a cut-off function $\eta \in C^\infty ( \mathbb{R} ) $ such that $0\leq \eta\leq 64$, $\mathrm{supp} (\eta) =\left[ 0,\frac{1}{4} \right] $, $ \int_0^1 \eta =4 $, $ \eta \geq 16 $ on $ \left[ \frac{1}{16} , \frac{3}{16} \right] $, and $ \eta'\leq 0 $ on $ \left[ \frac{1}{8} , \frac{1}{4} \right] $. Let $r_1 = \frac{1}{4} \int_0^{\frac{1}{4}} \left( \frac{1}{4} -s \right) \eta (s) ds \in \left( 0, \frac{1}{4} \right) $, and
\begin{eqnarray}\label{constructionphi}
\phi (r) = \left\{
\begin{aligned}
4r \;\;\;\;\;\;\;\;\;\;\;\;\;\;\;\;\;\;\; &, \; r\in [0,r_1] ;\\
4r - \int_{r_1}^r \int_0^{t-r_1} \eta (s) ds dt \; &, \; r\in (r_1 ,\infty) . 
\end{aligned}
\right.
\end{eqnarray}
Then we have $\phi' (r) =4 $ on $[0,r_1 ]$, $\phi' (r) = 0 $ on $\left[ \frac{1}{4} +r_1 ,\infty \right) $, $r_1 \geq \frac{1}{32} $, and
$$ \phi (r) \leq \phi \left( \frac{1}{4} +r_1 \right) = 1+ 4r_1 - \int_0^{\frac{1}{4}} \int_0^{t } \eta (s) ds dt = 1+ 4r_1 - \int_0^{\frac{1}{4}} \left( \frac{1}{4} -s \right) \eta (s) ds =1 . $$

Let $\delta = e^{-100} \left( \int_{0 }^r \int_{0 }^t \eta \left( 2s -\frac{1}{8} -2r_1 \right) ds dt \right)^{-1} \leq 32 e^{-100} $, and let 
\begin{eqnarray}\label{constructionrho}
\rho (r) = 1 + \delta \int_{0 }^r \int_{0 }^t \eta \left( 2s -\frac{1}{8} -2r_1 \right) ds dt .
\end{eqnarray}
Hence $\rho (r) =1 $ on $\left[ 0,r_1 +\frac{1}{16} \right] $, $\rho' (r) = e^{-100} $ on $\left[r_1 + \frac{3}{16} ,\infty \right) $, and $\rho'' \geq 0$. Now we only need to show that $\Ric (g_{\rho ,\phi} ) \geq 0 $ for such $\rho $ and $\phi $. By Lemma \ref{lmmriccicurvature}, we see that $\Ric_{g_{\rho ,\phi} } \geq 0 $ if and only if $\Ric_{g_{\rho ,\phi}} (X_0 ,X_0 ) \geq 0 $, $\Ric_{g_{\rho ,\phi}} (X_1 ,X_1 ) \geq 0 $, $\Ric_{g_{\rho ,\phi}} (X_2 ,X_2 ) \geq 0 $ and $\Ric_{g_{\rho ,\phi}} (X_3 ,X_3 ) \geq 0 $.

Our proof will be divided into several parts according to the value of $r$.

\smallskip

\par {\em Part 1.} In this part, we will prove that $\Ric_{g_{\rho ,\phi} } \geq 0 $ when $r\in \left[ 0,r_1 +\frac{1}{16} \right] $.

When $r\in \left[ 0,r_1 +\frac{1}{16} \right] $, we have $\rho =1$. Hence
\begin{eqnarray*}
\Ric_{g_{\rho ,\phi}} \left(X_0 ,X_0 \right) & = & - (3\rho^{-1} \rho'' +\phi^{-1} \phi'' +2\rho^{-1} \phi^{-1} \rho'\phi' ) g_{\rho,\phi} (X_0 ,X_0) =  - \phi^{-1} \phi'' \\
& = & \phi^{-1} \eta (r-r_1 ) \geq 0 , \\
\Ric_{g_{\rho ,\phi}} \left(X_1 ,X_1 \right) & = & - [\rho^{-1} \rho'' +\phi^{-1} \phi'' +4\rho^{-1} \phi^{-1} \rho'\phi' - 2\rho^{-2} \phi^2 + 2\rho^{-2} (\rho')^{2} ] g_{\rho ,\phi} \left(X_1 ,X_1 \right) \\
& = & 2 \phi^4 - \rho^2\phi\phi'' = 2 \phi^4 + \rho^2\phi\eta (r-r_1 ) \geq 0 ,
\end{eqnarray*}
and
\begin{eqnarray*}
\Ric_{g_{\rho ,\phi}} \left(X_2 ,X_2 \right) & = & \Ric_{g_{\rho ,\phi}} \left(X_3 ,X_3 \right) \\
& = & [- \rho^{-1} \rho'' - \rho^{-1} \phi^{-1} \rho'\phi' +4\rho^{-2} - 2\rho^{-2} \phi^2 - 2\rho^{-2} (\rho')^{2} ] g_{\rho ,\phi} \left(X_3 ,X_3 \right) \\
& = & \left( 4\rho^{-2} - 2 \rho^{-2} \phi^2 \right) \rho^2 \geq 2. 
\end{eqnarray*}
It follows that if $r\in \left[ 0,r_1 +\frac{1}{16} \right] $, then $\Ric_{g_{\rho ,\phi} } \geq 0 $.

\smallskip

\par {\em Part 2.} In this part, we will prove that $\Ric_{g_{\rho ,\phi} } \geq 0 $ when $r\in \left[ r_1 +\frac{1}{16} , r_1 +\frac{3}{16} \right] $.

In this case, we have $1\leq \rho (r) \leq 2 $, $0\leq\rho' (r) \leq e^{-100} $, $0\leq \rho'' (r) = \delta \eta \left( 2s -\frac{1}{8} -2r_1 \right) \leq 64\delta $, $4r_1 \leq \phi (r) \leq 1 $, $0\leq \phi' (r) \leq 4 $, and $\phi'' (r) = -\eta (r-r_1 ) \leq -16 $. Then
\begin{eqnarray*}
\Ric_{g_{\rho ,\phi}} \left(X_0 ,X_0 \right) & = & - (3\rho^{-1} \rho'' +\phi^{-1} \phi'' +2\rho^{-1} \phi^{-1} \rho'\phi' ) g_{\rho,\phi} (X_0 ,X_0)  \\
& \geq & -3\cdot 64\delta + 16 - 2\cdot \frac{1}{4r_1} \cdot e^{-100} \cdot 4 \geq 16 - e^{-80} > 0 , \\
\Ric_{g_{\rho ,\phi}} \left(X_1 ,X_1 \right) & = & - [\rho^{-1} \rho'' +\phi^{-1} \phi'' +4\rho^{-1} \phi^{-1} \rho'\phi' - 2\rho^{-2} \phi^2 + 2\rho^{-2} (\rho')^{2} ] g_{\rho ,\phi} \left(X_1 ,X_1 \right) \\
& \geq & - \left( 64\delta -16 +4\cdot \frac{1}{4r_1} \cdot e^{-100} \cdot 4 - 2\rho^{-2} \phi^2 + 2e^{-200} \right) \rho^2 \phi^2 >0 ,
\end{eqnarray*}
and
\begin{eqnarray*}
\Ric_{g_{\rho ,\phi}} \left(X_2 ,X_2 \right) & = & \Ric_{g_{\rho ,\phi}} \left(X_3 ,X_3 \right) \\
& = & [- \rho^{-1} \rho'' - \rho^{-1} \phi^{-1} \rho'\phi' +4\rho^{-2} - 2\rho^{-2} \phi^2 - 2\rho^{-2} (\rho')^{2} ] g_{\rho ,\phi} \left(X_3 ,X_3 \right) \\
& \geq & \left( - 64\delta - \frac{1}{4r_1} \cdot e^{-100} \cdot 4 +4\rho^{-2} - 2\rho^{-2} - 2\rho^{-2} e^{-200} \right) \rho^2 > 0 . 
\end{eqnarray*}
It shows that $\Ric_{g_{\rho ,\phi} } \geq 0 $ for $r\in \left[ r_1 +\frac{1}{16} , r_1 +\frac{3}{16} \right] $.

\par {\em Part 3.} In this part, we will prove that $\Ric_{g_{\rho ,\phi} } \geq 0 $ when $r\in \left[ r_1 +\frac{3}{16} , r_1 +\frac{1}{4} \right] $.

Since $r\in \left[ r_1 +\frac{3}{16} , r_1 +\frac{1}{4} \right] $, we have $1\leq \rho (r) \leq 2 $, $\rho' (r)= e^{-100} $, $4r_1 \leq \phi (r) \leq 1 $, $\phi'' (r) \leq 0 $ and $\phi'''(r) = -\eta' (r-r_1 ) \geq 0 $. Hence $\phi'' (r) $ is increasing, and
$$0\leq \phi' (r) = -\int_{r}^{r_1 +\frac{1}{4}} \phi'' (t) dt \leq -\int_{r}^{r_1 +\frac{1}{4}} \phi'' (r) dt \leq -\frac{\phi'' (r)}{16} .$$
It follows that
\begin{eqnarray*}
\Ric_{g_{\rho ,\phi}} \left(X_0 ,X_0 \right) & = & - (3\rho^{-1} \rho'' +\phi^{-1} \phi'' +2\rho^{-1} \phi^{-1} \rho'\phi' ) g_{\rho,\phi} (X_0 ,X_0)  \\
& \geq & - \left( \phi'' +2\cdot \frac{1}{4r_1} \cdot e^{-100} \phi' \right) \geq - \left( \phi'' + \phi' \right) \geq 0 , \\
\Ric_{g_{\rho ,\phi}} \left(X_1 ,X_1 \right) & = & - [\rho^{-1} \rho'' +\phi^{-1} \phi'' +4\rho^{-1} \phi^{-1} \rho'\phi' - 2\rho^{-2} \phi^2 + 2\rho^{-2} (\rho')^{2} ] g_{\rho ,\phi} \left(X_1 ,X_1 \right) \\
& \geq & - \left( \phi'' +4\cdot \frac{1}{4r_1} \cdot e^{-100} \phi' - 2\rho^{-2} \cdot 16 r_1^2 + 2\rho^{-2} e^{-200} \right) \rho^2 \phi^2 \geq 0 ,
\end{eqnarray*}
and
\begin{eqnarray*}
\Ric_{g_{\rho ,\phi}} \left(X_2 ,X_2 \right) & = & \Ric_{g_{\rho ,\phi}} \left(X_3 ,X_3 \right) \\
& = & [- \rho^{-1} \rho'' - \rho^{-1} \phi^{-1} \rho'\phi' +4\rho^{-2} - 2\rho^{-2} \phi^2 - 2\rho^{-2} (\rho')^{2} ] g_{\rho ,\phi} \left(X_3 ,X_3 \right) \\
& \geq & \left( - \frac{1}{4r_1} \cdot e^{-100} \cdot 4 +4\rho^{-2} - 2\rho^{-2} - 2\rho^{-2} e^{-200} \right) \rho^2 > 0 . 
\end{eqnarray*}

\par {\em Part 4.} In this part, we will prove that $\Ric_{g_{\rho ,\phi} } \geq 0 $ when $r\in \left[ r_1 +\frac{1}{4} ,\infty \right] $.

By definition, $r\in \left[ r_1 +\frac{1}{4} ,\infty \right] $ implies that we have $ \rho (r) \geq 1 $, $\rho' (r)= e^{-100} $, and $ \phi (r) = 1 $.
It follows that
\begin{eqnarray*}
\Ric_{g_{\rho ,\phi}} \left(X_0 ,X_0 \right) & = & - (3\rho^{-1} \rho'' +\phi^{-1} \phi'' +2\rho^{-1} \phi^{-1} \rho'\phi' ) g_{\rho,\phi} (X_0 ,X_0) =0 \\
\Ric_{g_{\rho ,\phi}} \left(X_1 ,X_1 \right) & = & - [\rho^{-1} \rho'' +\phi^{-1} \phi'' +4\rho^{-1} \phi^{-1} \rho'\phi' - 2\rho^{-2} \phi^2 + 2\rho^{-2} (\rho')^{2} ] g_{\rho ,\phi} \left(X_1 ,X_1 \right) \\
& \geq & - \left( -2\rho^{-2} + 2e^{-200} \rho^{-2} \right) \rho^2 = 2 - 2e^{-200} >0 ,
\end{eqnarray*}
and
\begin{eqnarray*}
\Ric_{g_{\rho ,\phi}} \left(X_2 ,X_2 \right) & = & \Ric_{g_{\rho ,\phi}} \left(X_3 ,X_3 \right) \\
& = & [- \rho^{-1} \rho'' - \rho^{-1} \phi^{-1} \rho'\phi' +4\rho^{-2} - 2\rho^{-2} \phi^2 - 2\rho^{-2} (\rho')^{2} ] g_{\rho ,\phi} \left(X_3 ,X_3 \right) \\
& \geq & \left( 4\rho^{-2} - 2\rho^{-2} - 2\rho^{-2} e^{-200} \right) \rho^2 = 2 - 2e^{-200} >0 . 
\end{eqnarray*}
This completes the proof.
\end{proof}

As a consequence, we can prove the following Proposition.

\begin{prop}\label{propexamplesmoothing}
There exist a $4$-dimensional Riemannian manifold $(M^4 ,g) $ and a compact subset $K\subset M^4 $ such that $M^4$ is diffeomorphic to an open subset of $\mathbb{R}^4 $, and $ (M^4 \setminus K ,g) $ is isometric to the annulus $A_{e^{200},\infty }  \subset C(\mathbb{S}^3 /Q_8 ) $, where the metric on $\mathbb{S}^3 / Q_8 $ is given by scaling of a standard $3$-sphere, $ C(\mathbb{S}^3 /Q_8 ) $ is the metric cone over $ \mathbb{S}^3 /Q_8 $, and $A_{r_1 ,r_2 }$ is the annulus $B_{r_2} (o) \setminus B_{r_1} (o) $.
\end{prop}

\begin{proof}
At first, we can find a tubular neighborhood $U$ of $P = f( \mathbb{R}P^2 ) \subset \mathbb{R}^4 $ such that $U$ is diffeomorphic to the total space of the normal bundle $N_P$. By the argument above we see that there exist a Riemannian metric $g_{\rho,\phi}$ on $N_p$ and a compact subset $K=\{ r\leq 2 \}$ such that $ (N_p \setminus K ,g_{\rho ,\phi} ) $ is isometric to the annulus $A_{R,\infty }  \subset C(\mathbb{S}^3 /Q_8 ) $ for some $R\leq 2e^{100}$, where the metric on $\mathbb{S}^3 / Q_8 $ is given by $e^{-100} g_{\mathbb{S}^3} $, $g_{\mathbb{S}^3}$ is the standard metric on $3$-sphere, $ C(\mathbb{S}^3 /Q_8 ) $ is the metric cone over $ \mathbb{S}^3 /Q_8 $, and $A_{r_1 ,r_2 }$ is the annulus $B_{r_2} (o) \setminus B_{r_1} (o) $. This is our assertion.
\end{proof}

\section{A sequence of topological torus }
\label{sectionsequencetorus}

At first, we prove Proposition \ref{propRnmetric}. For convenience, we restate Proposition \ref{propRnmetric} here.

\begin{prop}[=Proposition \ref{propRnmetric}]\label{propexamplesequencetorus}
Let $n\geq 4$. Then there exists a sequence of smooth Riemannian metrics $g_i $ on $\mathbb{R}^n$ such that:
\begin{enumerate}[1)]
\item $g_i = g_{\rm Euc} $ outside the Euclidean unit ball $B_1 (0) \subset \mathbb{R}^n $,
\item $\Ric_{g_i} \geq -\Lambda $ and $ {\rm diam} \left( B_1 (0) ,g_i \right) \leq D $ for some $\Lambda ,D>0$ independent of $i$,
\item The pointed Gromov-Hausdorff limit $(X,d)$ of $(\mathbb{R}^n ,g_i) $ is a topological orbifold but not a topological manifold.
\end{enumerate}
Furthermore, every singular point of $X$ has a neighborhood that is homeomorphic to $ (\mathbb{R}^4 /Q_8 ) \times \mathbb{R}^{n-4} $.
\end{prop}

\begin{proof}
Since $\mathbb{S}^{n-4} \times \mathbb{R}^{4} $ is diffeomorphic to a bounded domain in $\mathbb{R}^n$, one can assume that $n=4$. Before discussing technical details, we would like to explain the idea. In Section \ref{sectionconesmoothing}, we constructed a metric with non-negative Ricci curvature on the normal bundle of $\mathbb{R}P^2$ in $\mathbb{R}^4$, which is asymptotic to a smooth cone $C(\mathbb{S}^3_\delta /Q_8)$ for some $\delta >0$, where $\mathbb{S}_\delta^3 $ is $\mathbb{S}^3$ equipped with a Riemannian metric of constant sectional curvature $\delta^{-2}$. Now, we directly apply this result to the tubular neighborhood of $\mathbb{R}P^2$, to construct a sequence of metrics such that the tubular neighborhood converges in the Gromov-Hausdorff sense to a cone, while $\mathbb{R}P^2$ collapses to a point.

By Proposition \ref{propexamplesmoothing}, we can find a $4$-dimensional Riemannian manifold $(M^4 ,g) $ and a compact subset $K\subset M^4 $ such that there exist a diffeomorphism $F_1: M^4\to U \subset \mathbb{R}^{4} $ and an isomorphism $F_2 : (M^4 \setminus K ,g) \to (A_{e^{200},\infty } ,dr^2 + r^2 g_{\mathbb{S}^3 /Q_8 } ) \subset (C(\mathbb{S}^3 /Q_8 ) ,d) $, where $ (C(\mathbb{S}^3 /Q_8 ) ,d) $ is a metric cone over $(\mathbb{S}^3 /Q_8 ,g_{\mathbb{S}^3 /Q_8 } ) $, $g_{\mathbb{S}^3 /Q_8 } $ is a smooth metric on $ \mathbb{S}^3 /Q_8 $, and $A_{r_1 ,r_2 }$ is the annulus $B_{r_2} (o) \setminus B_{r_1} (o) $.

Clearly, for any $R_1 >R_2 \geq e^{200}$, we can find a diffeomorphism 
$$F_{R_1 ,R_2} : K\cup F_2^{-1} (A_{e^{200} ,e^4 R_1}) \to K\cup F_2^{-1} (A_{e^{200} ,e^4 R_2}) ,$$
such that 
$$ F_2 \circ F_{R_1 ,R_2} \circ F_2^{-1} : A_{R_1 ,e^4 R_1} \to A_{R_2 ,e^4 R_2} $$
is the standard transformation $(r,w)\mapsto (R_1^{-1} R_2 r,w ) $ on the metric cone $ (C(\mathbb{S}^3 /Q_8 ) ,d) $.

Now we consider the Riemannian manifolds $(M^4 , \epsilon^2 g) $ for any $\epsilon \in (0,1) $. Then 
$$g_{U,\epsilon} = \epsilon^2 (F^{-1}_{\epsilon^{-1} e^{201} ,e^{201}} \circ F^{-1}_1 )^{*} g $$ 
gives a Riemannian metric on $F_1 (K\cup F^{-1}_2 ( A_{e^{200} ,e^{205} } ) ) $, and one can see that the restriction of $g_{U,\epsilon}$ on $F_1 \circ F^{-1}_2 ( A_{e^{201} ,e^{205} } ) $ is independent of $\epsilon $. Let $ \eta $ be a smooth function on $\mathbb{R}^{4} $ such that $0\leq \eta \leq 1$, $\eta =1 $ on $F_1 (K\cup F^{-1}_2 ( A_{e^{200} ,e^{204.1} } ) ) $, and $\eta =0 $ outside $F_1 (K\cup F^{-1}_2 ( A_{e^{200} ,e^{204.9} } ) ) $. Fix a smooth Riemannian metric $g_T $ on $\mathbb{R}^{4} $. Set $g_{\epsilon } = \eta g_{U,\epsilon} + (1-\eta ) g_T $. Then the metrics $(\mathbb{R}^{4} ,g_{\epsilon } ) $ have a Ricci lower bound independent of $\epsilon$. Since the Riemannian manifold $(F^{-1}_2 ( A_{ e^{200} , \epsilon^{-1 } R_1 } ) ,\epsilon^2 g) $ is isomorphic to $ (A_{  e^{200} , \epsilon^{-1 } R_1 }  ,\epsilon^2 dr^2 + \epsilon^2 r^2 g_{\mathbb{S}^3 /Q_8 } ) \cong (A_{ \epsilon e^{200} ,R_1 } ,dr^2 + r^2 g_{\mathbb{S}^3 /Q_8 } )$, we see that the diameters ${\rm diam}_{g_\epsilon } ( K\cup F^{-1}_2 ( A_{e^{200} , \epsilon^{-1}  e^{202} } ) ) $ have an upper bound independent of $\epsilon$. Let $\epsilon \to 0$. By choosing subsequence if necessary, we can find $\epsilon_i \to 0 $, and the manifolds $(N^4 ,g_{\epsilon_i } ) $ converges to a metric space $(X,d_X)$ as $i\to\infty$. Let $d_i $ be the metric on $\mathbb{R}^{4} $ given by $g_{\epsilon_i}$, and $d_{M ,i}$ be the metric on $M^4$ given by $\epsilon^2_i g $.

Clearly, we have $F_1^* d_{i } = d_{M ,i}  $ on $K\cup F^{-1}_2 ( A_{ e^{200} , e^{202 } } ) $. Then $\left( K\cup F^{-1}_2 ( A_{ e^{200} , e^{202 } } ) , F_1^* d_i \right) $ converge to $ B_{e^{202}} (o) \subset (C(\mathbb{S}^3 /Q_8 ) ,d) $ as $i\to\infty$. Note that for any $x\in K$, the pointed Riemannian manifolds $(M^4 ,\epsilon_i^2 g ,x)$ converge to $(C(\mathbb{S}^3 /Q_8 ) ,d ,o) $ in the pointed Gromov-Hausdorff sense. Since the restriction of $g_{\epsilon_i} $ on $\mathbb{R}^{4} \setminus F_1 ( K\cup F^{-1}_2 ( A_{ e^{200} , e^{201 } } ) $ is independent of $i$ for sufficiently large $i$, we see that the restriction of $d_{i} $ on $\mathbb{R}^{4} \setminus F_1 ( K\cup F^{-1}_2 ( A_{ e^{200} , e^{202 } } ) $ is also independent of $i$ for sufficiently large $i$, where $A_{r_1 ,r_2 }$ is the annulus $B_{r_2} (o) \setminus B_{r_1} (o) $. 

Now we can find a point $y$ in $X$ such that there exists an open neighborhood of $y$ homeomorphic to $C(\mathbb{S}^3 /Q_8 ) $. Assume that the image of $y$ in $C(\mathbb{S}^3 /Q_8 ) $ is the vertex. Then for any open neighborhood $U$ of $y$, $\pi_1 (U \setminus \{ y\} ) \neq 0 $. It follows that $X$ is not a topological manifold.
\end{proof}

As a consequence, we can prove Theorem \ref{thmexamplesequencetorus}.

\vspace{0.2cm}

\noindent \textbf{Proof of Theorem \ref{thmexamplesequencetorus}: }
Let $(\mathbb{R}^{4} ,\tilde{g}_i)$ be the sequence of Riemannian metrics in Proposition \ref{propexamplesequencetorus}, and let $ g_{T^j} $ be the flat metric on $T^j$ given by standard $\mathbb{R}^j /\mathbb{Z}^j $. Since $\tilde{g}_i = g_{\rm Euc} $ outside the Euclidean unit ball $B_1 (0) \subset \mathbb{R}^n $, we see that $\tilde{g}_i$ gives a sequence of Riemannian metrics on the torus $T^4 \cong [-D-2 ,D+2]^4 / \thicksim $, and the Gromov-Hausdorff limit $(X,d)$ of $(T^{4} ,\tilde{g}_i)$ is a topological orbifold, and there exists an open neighborhood of $x\in X$ homeomorphic to $C(\mathbb{S}^3 /Q_8)$.

Then the sequence of Riemannian manifolds 
$$(T^4 \times T^{k-4} \times T^{n-k} ,g_i + g_{T^{k-4}} + i^{-1} g_{T^{n-k}} )$$ 
converges to $ X\times T^{k-4} $ in the Gromov-Hausdorff sense, which proves this Theorem.
\qed

\section{\texorpdfstring{$4$}{Lg}-dimensional case }
\label{4dimcase}

Let $(M_i^4 , g_i) \xrightarrow{GH} (X^k,d_X)$ satisfy ${\rm Ric}_{g_i} \ge -3$ and $ {\rm diam} (M_i^4 , g_i) \le D $, where $D>0$ is a constant independent of $i$, and $k\leq 4$ is the rectifiable dimension of $X$. Assume that $M^4_i$ are all homeomorphic to tori $T^4$.

\subsection{Ricci curvature bounded case}

At first, we recall the topological classification of spherical hypersurfaces in $\mathbb{R}^4$.

\begin{thm}[{\cite[Theorem 2.2]{crisphillman1}}]\label{sphericalhyper}
Let $\Gamma \leq O(4)$ be a discrete subgroup acting freely on $\mathbb{S}^3$. If there exists a locally flat embedding of $\mathbb{S}^3 / \Gamma $ into $\mathbb{S}^4 $, then $\Gamma\cong \{ \mathrm{id} \} $, $ Q_8 $ or $ I^* $, where $ Q_8 $ is the quaternion group, and $I^*$ is the binary icosahedral group. Here locally flat means that for any point $p$ of the image $N$, there exists an open neighborhood $U$ of $p$ such that the pair $(U,U\cap N)$ is homeomorphic to $(\mathbb{R}^4,\mathbb{R}^3 )$. Moreover, if the embedding is smooth, then $\Gamma\cong \{ \mathrm{id} \} $ or $ Q_8 $.
\end{thm}

Now we consider the non-collapsing limit of $4$-dimensional tori with bounded Ricci curvature. Our argument relies on the following fact, which is known to experts.

\begin{fact}
\label{factresolutionofbiggergroup}
Let $(M^4_i, g_i ,p_i) \xrightarrow{GH} (X ,d_X ,p_\infty)$ satisfy $\left| {\rm Ric}_{g_i} \right| \le 3 $ and $ {\rm Vol} (B_1 (p_i) ) \geq v $, where $v>0$ is a constant independent of $i$. Assume that $X$ is an orbifold with a singular point of the type $\mathbb{R}^4/\Gamma $ for some non-trivial discrete subgroup $ \Gamma \leq  O(4)$ acting freely on $\mathbb{S}^3$. Then for sufficiently large $i$, there exist an ALE Ricci flat $4$-manifold $(V,g_V)$ asymptotic to $ \mathbb{R}^4/\Gamma_1 $ and a $C^{2,\frac{1}{2}}$ embedding $\iota : V\to M_i $, where $\Gamma_1$ is a non-trivial discrete subgroup of $ O(4)$ acting freely on $\mathbb{S}^3$, and $|\Gamma_1|\leq |\Gamma|$, and $|G|$ is the order of $G$.
\end{fact}

\begin{proof}
By the standard result about non-collapsing limit of $4$-dimensional manifold with $2$-side Ricci bound due to Anderson \cite{anderson1}, Bando–Kasue–Nakajima \cite{bkn1} and Tian \cite{tg5}, $X$ is a topological orbifold with isolated singularities, and the metrics $g_i$ converging to a $C^{1,\frac{1}{2}}$ metric on the regular part of $X$. Let $x\in X$ denote the singular point of the type $\mathbb{R}^4/\Gamma $, and choose a sequence $x_i \in M_i $ converging to $x$. Then one can find a neighborhood $U$ of $x$ and a $C^{2,\frac{1}{2}}$-diffeomorphic $ F : U \to \mathbb{R}^4 /\Gamma $. Now we consider the standard embedding $\Phi : \mathbb{S}^3 / \Gamma \to \mathbb{R}^4 /\Gamma $. One can see that the image $\Sigma = F^{-1} ( \Phi ( \mathbb{S}^3 / \Gamma )) $ divides $X$ in to two parts, $X_1$ and $X_2$, and $\Sigma$ is there boundary. Assume that $x\in X_1$. Then there exists $\epsilon >0$ such that $B_{5\epsilon} (x) \in U \cap X_1$. Hence one can find a $C^{2,\frac{1}{2}}$-embedding $\Sigma \to \Sigma_i $ for sufficiently large $i$ by the convergence of the differential structure of $M_i $ outside the topological singularities of $X$. Clearly, for sufficiently large $i$, $\Sigma_i $ divides $M_i$ in to two parts, $M_{i,1}$ and $M_{i,2}$, and $B_{\epsilon} (x_i) \subset M_{i,1} $. Let 
$$r_i = \inf \left\lbrace r\in (0,\epsilon) : \Vol \left( B_r (y) \right) \geq \frac{2}{3} \Vol \left( B^{\mathbb{R}^4}_r (0) \right) , \forall y\in B_{2\epsilon } (x_i ) \right\rbrace .$$
Clearly, $|\Gamma |\geq 2$ shows that $\lim_{i\to\infty } r_i =0 $. Since $x$ is the unique singular point in $B_{5\epsilon } (x)$, we see that for sufficiently large $i$, there exists $y_i \in B_{\epsilon} (x_i) $ such that $\Vol \left( B_{r_i} (y_i) \right) = \frac{2}{3} \Vol \left( B^{\mathbb{R}^4}_{r_i} (0) \right) $. Let $(Y,d_y ,y)$ be the Gromov-Hausdorff limit of the pointed manifolds $(M_i , r^{-2}_i g_i ,y_i ) $. Then $Y$ is an orbifold, and we have $\Vol (B_{1} (y) ) = \frac{2}{3} \Vol \left( B^{\mathbb{R}^4}_{1} (0) \right) $. Hence $(M_i , r^{-2}_i g_i ,y_i ) $ converge to $(Y,d_y ,y)$ in the $C^{1,\frac{1}{2}}$ sense, and $d_Y$ is induced by a $C^{1,\frac{1}{2}}$ metric $g_Y$ on $Y$. Since $r_i \to 0 $, one can see that $\Ric (g_Y) =0 $ in the weak sense. Then $g_Y$ is smooth and $\Ric (g_Y) =0 $. It follows that there exists an asymptotically locally Euclidean Ricci flat manifold $(V,g_V )$ with the end $\mathbb{S}^3 /\Gamma_1 $ which embed in $\tilde{M}_i$ smoothly for sufficiently large $i$. Hence one can construct the embedding $\iota $. Note that $\Vol (B_{1} (y) ) = \frac{2}{3} \Vol \left( B^{\mathbb{R}^4}_{1} (0) \right) $ implies that $\Gamma_1$ is non-trivial.
\end{proof}

Now we are ready to prove Proposition \ref{propT42-sidericcibound} in the non-collapsed case.

\begin{lmm}\label{lmmT4noncollap2-sidericcibound}
Let $(M^4_i, g_i) \xrightarrow{GH} (X ,d_X)$ satisfy $\left| {\rm Ric}_{g_i} \right| \le 3 $ and $ {\rm diam} (M_i ,g_i ) \le D $, where $D>0$ is a constant independent of $i$, and the rectifiable dimension of $X$ is $4$.  Assume that $M^4_i$ are all homeomorphic to tori $T^4$ or $\mathbb{S}^4$. Then $X $ is homeomorphic to $M_i$ for sufficiently large $i$.
\end{lmm}

\begin{proof}
Assume that $x\in X$ is singular point of the type $\mathbb{R}^4/\Gamma $, where $ \Gamma $ is a non-trivial discrete subgroup of $O(4)$ acting freely on $\mathbb{S}^3$. Then by Fact \ref{factresolutionofbiggergroup}, one can see that for sufficiently large $i$, there exist an ALE Ricci flat $4$-manifold $(V,g_V)$ asymptotic to $ \mathbb{R}^4/\Gamma_1 $ and a $C^{2,\frac{1}{2}}$ embedding $\iota : V\to M_i $, where $\Gamma_1$ is a non-trivial discrete subgroup of $ O(4)$ acting freely on $\mathbb{S}^3$. Now we fix a sufficiently large $i$. Since $M_i$ are all homeomorphic to $T^4$ or $\mathbb{S}^4$, the universal covering spaces $\tilde{M}_i$ are all homeomorphic to $\mathbb{R}^4$ or $\mathbb{S}^4$. Hence, one can find a locally flat embedding of $ \mathbb{S}^3/\Gamma_1 $ into $\mathbb{S}^4$. By Theorem \ref{sphericalhyper}, one can see that $ \mathbb{S}^3/\Gamma_1 $ is isometric to $\mathbb{S}^3$, $\mathbb{S}^3 / Q_8 $ and $\mathbb{S}^3 / I^* $, where $I^*$ is the binary icosahedral group. Since $x$ is a singular point, $\Gamma_1 \cong Q_8 $ or $ I^* $.

Clearly, the volume growth of $V$ shows that $\pi_1 (V) $ is finite, the connectedness of $V$ gives $H_0 (V;\mathbb{Z} ) \cong \mathbb{Z} $, and the openness of $V$ implies that $H_4 (V;\mathbb{Z}) =0 $. Hence the Betti numbers $ b_1 (V) =b_4 (V) =0 $, and $b_0 (V) =1$. Fix an open subset $V' \subset \iota(V) $ with smooth boundary $\partial V' \cong \mathbb{S}^3 /\Gamma_1 $ such that $\iota(V) \setminus V' $ diffeomorphic to $\left( \mathbb{R}^4 \setminus B_1 (0) \right) /\Gamma_1  $. Now we apply the Mayer-Vietoris exact sequence to $\bar{V'} $ and $\tilde{M}_i \setminus V' $. Then we get the exact sequence
$$ H_2 (\mathbb{S}^3 /\Gamma_1 ;\mathbb{Q} ) \to H_2 (V' ;\mathbb{Q} ) \oplus H_2 (\tilde{M}_i \setminus V' ;\mathbb{Q} ) \to H_2 (\tilde{M}_i ; \mathbb{Q} ) .$$
Since $\tilde{M}_i $ is homeomorphic to $\mathbb{R}^4$ or $\mathbb{S}^4 $, we have $H_2 (\tilde{M}_i ; \mathbb{Q} ) =0 $. By the Poincar\'e duality, we have $H_2 (\mathbb{S}^3 /\Gamma_1 ;\mathbb{Q} ) \cong H_1 (\mathbb{S}^3 /\Gamma_1 ;\mathbb{Q} ) = 0 $. Hence $b_2 (V) =b_2 (V') =0 $, the signature $\tau (V) =0 $, and the Euler number $\chi (V) = b_0 (V) - b_1 (V) +b_2 (V) -b_3 (V) +b_4 (V) = 1-b_3 (V) \leq 1 $.

Now we apply the Hitchin's inequality on $V$. Then we have
$$ 2 > 2 \left( \chi (V) - \frac{1}{|\Gamma_1 |} \right) \geq 3 \left| \tau (V) + \eta \left( \mathbb{S}^3 / \Gamma_1 \right) \right| = 3 \left| \eta \left( \mathbb{S}^3 / \Gamma_1 \right) \right| ,$$
where $\eta \left( \mathbb{S}^3 / \Gamma_1 \right) $ is the eta invariant of the space $\mathbb{S}^3 / \Gamma_1 $ for the signature operator.

Since $\Gamma_1 \cong Q_8 $ or $I^*$, a direct calculation gives $\eta \left( \mathbb{S}^3 / \Gamma_1 \right) = \pm \frac{3}{4} $ or $\pm \frac{361}{180}$ (up to the orientation). See also \cite{nakajima1}. Hence $3 \left| \eta \left( \mathbb{S}^3 / \Gamma_1 \right) \right| \geq \frac{9}{4} >2 $, contradiction. Then we see that the singular set of $X$ is empty, and hence $M_i$ converge to $X$ in the $C^{1,\frac{1}{2}}$ sense. It follows that $M_i $ is homeomorphic to $X$ for sufficiently large $i$, which completes the proof.
\end{proof}

Before proving the collapsed case, we recall the theorem about uniform lattice coverings.

\begin{thm}[{\cite[Theorem 4.2]{ks1}}]\label{thmuniformlattice}
Let $(M^n,g)$ be homeomorphic to $T^n$ with ${\rm diam} (M,g) \le D$. Let $(\tilde M, \tilde g)$ be the universal cover where $\Gamma= \pi_1(M)$ acts by isometries. Then there exists $\Lambda\leq \Gamma$ with $\Lambda \cong \mathbb{Z}^n$ and 
	\begin{itemize}
		\item[(i)] ${\rm diam} (\tilde M/\Lambda , \tilde{g}) \le 6^n D$;

		\item[(ii)] There exists a family of generators $\{\gamma_1, \ldots , \gamma_n\}$ of $\Lambda$ such that
		\begin{equation*}
			A(n)^{-1} D \,\| \gamma \|_1 \le \dist_{\tilde M}(x,\gamma(x)) \le A(n) D\, \| \gamma \|_1 \, , \quad \forall \, \gamma \in \Lambda\, , x\in \tilde M \, ,
		\end{equation*}
where $\| \gamma_1^{\alpha_1} \circ \ldots \circ \gamma_n^{\alpha_n} \|_1 = \sum_{i=1}^n |\alpha_i |$.
	\end{itemize}
\end{thm}

Then we consider the collapsing case. Here we use an argument similar to the argument in \cite{bruenabersemola1}.

\begin{lmm}\label{lmmT4collapsing}
Let $(M_i^4 , g_i) \xrightarrow{GH} (X^k,d_X)$ satisfy $|{\rm Ric}_{g_i}| \le 3 $ and $ {\rm diam} (M_i^4 , g_i) \le D $, where $D>0$ is a constant independent of $i$, and $k\leq 3$ is the rectifiable dimension of $X$. Assume that $M_i$ are all homeomorphic to tori $T^4$. Then $X $ is homeomorphic to $T^k$, and $M_i$ is diffeomorphic to the standard $T^4$ for sufficiently large $i$.
\end{lmm}

\begin{proof}
Let $\Gamma_i = \pi_{1} (M_i ) \leq {\rm Iso} (\tilde{M}_i ,g_i ) $. We consider the pointed equivariant Gromov-Hausdorff convergence $(\tilde{M}_i ,g_i ,x_i ,\Gamma_i ) \to (\tilde{X}_\infty ,d_X ,x_\infty ,\Gamma_\infty ) $, where $x_i \in M_i $. Let $\Lambda_i \leq \Gamma_i $ be the uniform lattices in Theorem \ref{thmuniformlattice}. Without loss of generality, we can assume that $(\tilde{M}_i ,g_i ,x_i ,\Lambda_i )$ converging to $(\tilde{X}_\infty ,d_X ,x_\infty ,\Lambda_\infty )$ in the pointed equivariant Gromov-Hausdorff sense. By Lemma \ref{lmmT4noncollap2-sidericcibound}, we see that $\tilde{X}_\infty / \Lambda_\infty $ is homeomorphic to $T^4$, and the metric is given by a $C^{1,\frac{1}{2}}$ Riemannian metric.

By Colding-Naber's result \cite[Theorem 1.21]{coldingnaber1}, ${\rm Iso} (\tilde{X}_\infty )$ and ${\rm Iso} (\tilde{X}_\infty /\Lambda_\infty )$ are Lie groups. Since $\Gamma_\infty$, $\Lambda_\infty $ are closed subgroups of ${\rm Iso} (\tilde{X}_\infty )$, and $\Gamma_\infty/\Lambda_\infty $ is a closed subgroup of ${\rm Iso} (\tilde{X}_\infty /\Lambda_\infty )$, one can see that they are also Lie groups. Let $G\leq \Gamma_\infty $ be a compact subgroup. Since $\Gamma_\infty $ is abelian, we see that the elements in $G$ and $\Lambda_\infty $ are commutative. It follows that the diameters of orbits of $G$ is uniformly bounded. Since $\tilde{X}_\infty / \Gamma_\infty $ is homeomorphic to $T^4$, we see that there exists an action of $G$ on $\mathbb{R}^4$ such that the diameters of orbits of $G$ is uniformly bounded under the usual Euclidean metric. Hence $G=\{{\rm id}\}$ \cite[Chapter III, Corollary 9.7]{bredon1}, and the action of $\Gamma_\infty $ on $\tilde{X}_\infty $ is free. 

By the compactness of $ \tilde{X}_\infty /\Lambda_\infty $, we see that $\Gamma_\infty/\Lambda_\infty \cong \Gamma_0 \oplus T^a $, where $\Gamma_0 $ is a finite abelian group, and $a\in\mathbb{N}$. Let $\Gamma'_0$ be the preimage of $\Gamma_0 $ under the homomorphism $\Gamma_\infty \to \Gamma_\infty /\Lambda_\infty $. Clearly, $\Gamma'_0$ is discrete. Then $\tilde{X}_\infty / \Gamma'_0 $ is homeomorphic to $T^4$. Note that an aspherical closed manifold with an abelian fundamental group is homeomorphic to a torus. See Freedman and Quinn's book \cite{frequi1} for the $4$-dimensional case. Write $\tilde{X}' = \tilde{X}_\infty / \Gamma'_0 $.

Now we consider the action of $T^a$ on $\tilde{X}' $. Choose $\gamma_1 \in T^a $ such that $\gamma_1 (x') = x' $ for some $x'\in \tilde{X}'$. Let $\tilde{\gamma}_1 \in \Gamma_\infty $ such that the image of $\tilde{\gamma}_1 $ under the quotient homomorphism $\Gamma_\infty \to \Gamma_\infty /\Gamma'_0 $ is $\gamma_1 $, and let $\tilde{x} \in \tilde{X}_\infty $ such that the image of $\tilde{x} $ under the quotient map $\tilde{X}_\infty \to \tilde{X}_\infty /\Gamma'_0 $ is $x' $. Then we can find $\gamma_2 \in \Gamma'_0 $ such that $\gamma^{-1}_2 \circ \tilde{\gamma}_1 (\tilde{x})=\tilde{x} $. Since the action of $\Gamma_\infty $ on $\tilde{X}_\infty $ is free, we see that $\gamma_2 = \tilde{\gamma}_1 $, and hence $\gamma_1 ={\rm id}$. It follows that the action of $T^a$ on $\tilde{X}' $ is free. By the standard theory, the action $ T^a \times \tilde{X}' \to \tilde{X}' $ is $C^{2,\frac{1}{2}}$. See \cite[Theorem 4]{bochmon1} and \cite{calabihartman1} for more details. Then there exists a unique $C^{1,\frac{1}{2}}$ Riemannian structure on the quotient $\tilde{X}' /T^a $ such that the projection $ \tilde{X}' \to \tilde{X}' /T^a $ is a $C^{1,\frac{1}{2}}$ Riemannian submersion. It follows that $ \tilde{X}' $ is $C^{2,\frac{1}{2}}$-diffeomorphic to a smooth $T^a$-bundle over $ X= \tilde{X}' /T^a $. Since the projection $ \tilde{X}' \to \tilde{X}' /T^a $ is a topological trivial bundle \cite[Chapter IV, Theorem 9.5]{bredon1}, we see that the homotopy groups of $X $ is isomorphic to the homotopy groups of $T^{4-a}$, and hence $X$ is homeomorphic to $T^k =T^{4-a}$.

By Fukaya’s fibration theorem \cite{kf1, jckfmg1}, we see that for sufficiently large $i$, $M_i$ is diffeomorphic to a fiber bundle over $T^k$ with infra-nilmanifold fiber $F_i$. See also \cite[Theorem 0.6]{jzh1}. By the long exact sequence in the homotopy groups:
$$\cdots \to \pi_{k} (F_i) \to \pi_{k} (M_i) \to \pi_{k} (T^k) \to \pi_{k-1} (F_i) \to \cdots \to \pi_{1} (T^k) \to \pi_{0} (F_i) = \{ e \} ,$$
we have $\pi_1 (F_i) \cong \mathbb{Z}^{4-k} $ and $\pi_i (F_i) =0$, $\forall i\geq 2 $. Note that for any torus $T^m$, we have $\pi_1 (F_i) \cong \mathbb{Z}^{m} $ and $\pi_i (F_i) =0$, $\forall i\geq 2 $. Hence the bundle $F_i \to M_i \to T^k $ is topological trivial. If $4-k\leq 3$, one can see that $F_i$ is diffeomorphic to standard $T^{4-k}$ \cite{mt1, p1, p2, p3}.

When $k=0$, we see that ${\rm diam} (M_i ,g_i ) \to 0$ as $i\to \infty$. Write $\lambda_i = {\rm diam} (M_i ,g_i ) $. Then we can assume that $  (M_i ,\lambda^{-2}_i g_i ) $ converge to $ (Y^{k'} ,d_{Y} )$ in the Gromov-Hausdorff sense, where $k'\geq 1$ is the rectifiable dimension of $Y$. If $k'=4$, then $ (Y ,d_{Y} )$ is a flat torus, and hence $ M_i $ is diffeomorphic to the standard $T^4$ for sufficiently large $i$. Now we can replace $(M_i ,g_i )$ by $(M_i , \lambda^{-2}_i g_i )$, and reduce the question to the case $k\in \{ 1,2,3 \}$.

When $k=1$, $ M_i $ is diffeomorphic to a smooth $T^3$-bundle over $ S^1 $, and this bundle is topological trivial. Let $\{U_1 ,U_2\}$ be the standard open covering of $S^1$ by intervals. Then the connected components of $U_1 \cap U_2 $ is $V_1$, $V_2$. Now we consider the trivializations $\phi_1 : U_1 \times T^3 \to M_i $ and $\phi_2 : U_2 \times T^3 \to M_i $. Then $\phi^{-1}_1 \circ \phi_2 $ gives a smooth map $\varphi_{12,s} : U_1 \cap U_2 \to {\rm Diff} (T^3) $. Since $M_i$ is a topological trivial bundle, we can find a continuous map $\varphi_{2,h} : U_2 \to {\rm Homeo} (T^3) $ such that $ \varphi_{2,h} |_{U_1 \cap U_2 } = \varphi_{12,s} $. Hence the automorphism given by $\varphi_{12,s} (t) $ is independent of $t\in U_1\cap U_2$. Since the natural homomorphism from the mapping class group of $T^3$ to the outer automorphism of its fundamental group is an isomorphism \cite{wal1}, we see that there exists a smooth map $\varphi_{2,s} : U_2 \to {\rm Diff} (T^3) $ such that $ \varphi_{2,s} |_{U_1 \cap U_2 } = \varphi_{12,s} $. Hence $ M_i $ is diffeomorphic to the smooth trivial $T^3$-bundle over $ S^1 $.

When $k=2,3$, $ M_i $ is diffeomorphic to a smooth $T^{4-k}$-bundle over $ T^k $, and this bundle is topological trivial. Then we consider the projection $M_i \to S^1 $ given by the composition of $M_i \to T^k $ and $T^k \to S^1 $. By the argument in the case $k=1$, we see that $ M_i $ is diffeomorphic to the standard $T^4$. It follows that $ M_i $ is diffeomorphic to the standard $T^4$ for $k\leq 3$ and sufficiently large $i$, which completes the proof.
\end{proof}

\subsection{K\"ahler case}
Now we consider the non-collapsing case.

\begin{lmm}\label{lmmT4noncollapsingkahler}
Let $(M_i^4 , g_i) \xrightarrow{GH} (X^k,d_X)$ satisfy ${\rm Ric}_{g_i} \ge -3$ and $ {\rm diam} (M_i^4 , g_i) \le D $, where $D>0$ is a constant independent of $i$. Assume that $M^4_i$ are all homeomorphic to tori $T^4$, the rectifiable dimension of $X$ is $4$, and $g_i$ are K\"ahler.

Then $X $ is a topological orbifold with isolated singularities. Moreover, every point of $X$ has a neighborhood that is homeomorphic to an open subset of $ \mathbb{R}^4 /Q_8 $.
\end{lmm}

\begin{proof}
By Liu-Sz\'ekelyhidi's results in \cite{lggs1, lggs2}, we see that $X$ is a topological orbifold with isolated singularities when $g_i  $ are K\"ahler. See also Appendix \ref{ghpolarized2dim}. Note that the Albanese map implies that $M_i$ are K\"ahler tori, and K\"ahler metrics on $\mathbb{C}^n$ must be polarized. Moreover, the differential structure of $M_i $ converges to the differential structure of $X$ in the $C^\infty $ sense outside the topological singularities.

Let $x\in X$ be a topological singular point. Then we can find a neighborhood $U$ of $x$ homeomorphic to $\mathbb{R}^4 /\Gamma $ for some $\Gamma \leq O(4) $. Now we consider the embedding $\mathbb{S}^3 / \Gamma \to \mathbb{R}^4 /\Gamma $. Hence one can find a smooth embedding $\mathbb{S}^3 / \Gamma \to M_i $ for sufficiently large $i$ by the convergence of the differential structure of $M_i $ outside the topological singularities of $X$. Since $\Gamma $ is finite, we see that map $\mathbb{S}^3 \to \tilde{M}_i $ given by lifting property of universal covering gives a smooth embedding $\mathbb{S}^3 / \Gamma \to \tilde{M}_i $.

By \cite[Theorem 2.2]{crisphillman1}, one can see that the only spherical manifolds which locally flat embed in $\mathbb{R}^4 $ are $\mathbb{S}^3$, $\mathbb{S}^3 / Q_8 $ and $\mathbb{S}^3 / I^* $, and the embedding of $\mathbb{S}^3 / I^* $ in $\mathbb{R}^4$ cannot be smooth, where $I^*$ is the binary icosahedral group. Since $M_i$ are K\"ahler tori, $ \tilde{M}_i $ are diffeomorphic to $\mathbb{R}^4$, and $\Gamma \ncong I^* $. If $\Gamma \cong \{ e\} $, the $x$ is not a singular point in $X$.

It follows that $\Gamma \cong Q_8 $, and $\mathbb{R}^4 /\Gamma $ is homeomorphic to $\mathbb{R}^4 /Q_8 $.
\end{proof}

Then we consider the collapsing case. Here we use an argument similar to the argument in \cite{bruenabersemola1}.

\begin{lmm}\label{lmmkahlerT4collapsing}
Let $(M_i^4 , g_i) \xrightarrow{GH} (X^k,d_X)$ satisfy ${\rm Ric}_{g_i} \ge -3 $ and $ {\rm diam} (M_i^4 , g_i) \le D $, where $D>0$ is a constant independent of $i$, and $k\leq 3$ is the rectifiable dimension of $X$. Assume that $g_i$ are K\"ahler metrics and $M^4_i$ are all homeomorphic to tori $T^4$. Then $X $ is homeomorphic to $T^k$.
\end{lmm}

\begin{proof}
Let $\Gamma_i = \pi_{1} (M_i ) \leq {\rm Iso} (\tilde{M}_i ,g_i ) $. We consider the pointed equivariant Gromov-Hausdorff convergence $(\tilde{M}_i ,g_i ,x_i ,\Gamma_i ) \to (\tilde{X}_\infty ,d_X ,x_\infty ,\Gamma_\infty ) $, where $x_i \in M_i $. Let $\Lambda_i \leq \Gamma_i $ be the uniform lattices in Theorem \ref{thmuniformlattice}. Without loss of generality, we can assume that $(\tilde{M}_i ,g_i ,x_i ,\Lambda_i )$ converging to $(\tilde{X}_\infty ,d_X ,x_\infty ,\Lambda_\infty )$ in the pointed equivariant Gromov-Hausdorff sense. 

By Lemma \ref{lmmT4noncollapsingkahler}, $\tilde{X}_\infty /\Lambda_\infty $ is a topological orbifold with isolated singularities. Since the rectifiable dimension of $X=\tilde{X}_\infty /\Gamma_\infty$ is less than the rectifiable dimension of $\tilde{X}_\infty /\Lambda_\infty $, we see that for any $\delta >0$ and $x'\in X'$, we can find $\gamma_{x'} \in \Gamma_\infty $ such that $d_{\tilde{X}_\infty /\Lambda_\infty } (x' ,\gamma_{x'}) \in (0,\delta ) $. It follows that $\tilde{X}_\infty /\Lambda_\infty$ doesn't have any isolated singular point, and hence $\tilde{X}_\infty /\Lambda_\infty$ is a topological manifold. Then $\tilde{X}_\infty $ is a complex manifold and $\Gamma_\infty $ acts analytically on $\tilde{X}_\infty $. Now we can follow the argument in Lemma \ref{lmmT4collapsing} to show that $X $ is homeomorphic to $T^k$.
\end{proof}

\appendix
 \section{Gromov-Hausdorff limits of polarized surfaces are orbifolds} 
 \label{ghpolarized2dim}
 
 In this appendix, we show that the Gromov-Hausdorff limits of polarized surfaces are holomorphic orbifolds. This is a refinement in the case of K\"ahler surfaces of the result obtained by Liu-Sz\'ekelyhidi, who generalized Tian's partial $C^0$-estimates \cite{tg5, tg6, tg3} to K\"ahler tangent cones \cite{lggs2} and polarized K\"ahler manifolds \cite{lggs1} with Ricci lower bound.
 
 At first, we show that every tangent cone of $X$ at $x$ is a topological orbifold.

\begin{prop}
\label{propconeorbifold}
Let $\left( X ,d ,p \right) $ be the pointed Gromov-Hausdorff limit of a sequence of pointed complete K\"ahler surfaces $\left( M_i ,\omega_i ,p_i \right) $ with $\Ric \left( \omega_i \right) \geq -\Lambda \omega_i $ and $\Vol \left( B_1 \left( p_i \right) \right) \geq v $, $\forall i\in\mathbb{N} $, where $\Lambda ,v>0$ are constants. Assume that $(X,d,p) \cong ( C(Y) ,o) $ is a metric cone, where $(Y,d_Y)$ is a compact metric space. Then we can find a discrete subgroup $\Gamma \in U(2) $ which acts freely away from the origin of $\mathbb{C}^2 $, such that $X$ is homeomorphic to $\mathbb{C}^2 /\Gamma $.
\end{prop}

\begin{proof}
By \cite{lggs2}, $X$ is homeomorphic to an affine normal surface, and one can see that $X\setminus \{ p \} $ is homeomorphic to a smooth $4$-dimensional manifold. Note that $X$ is normal and for any $r, \epsilon >0 $, $\mathcal{H}^2 (B_r (p) \setminus \mathcal{R}_{\epsilon } ) < \infty $. Then the complex structure on $X$ can be determined by the complex structure on $\mathcal{R}_\epsilon $ \cite[Theorem 4.1]{harpol1}, and the scaling maps on $X$ are holomorphic maps. Now we consider the following action of a one-parameter group on $X\setminus \{ p \} $ given by the rescalings:
\begin{equation*}
F:\left\{
\begin{aligned}
\mathbb{R} \times (X\setminus \{ p \} ) \;\;\;\;& \longrightarrow \;\;\;\; X\setminus \{ p \} ,\\
(t ,r ,y)  \;\;\;\;\;\;\;\;\; &\longmapsto \;\;\;\; \left( e^t r,y \right) , \\
\end{aligned}
\right.  
\end{equation*}
where $r>0 $, $y\in Y$, and $(r,y) \in C(Y) \setminus \{o\} \cong X\setminus \{p\} $. Since the group action 
$F$ preserves the complex structure of $X\setminus \{ p \} $ and has no fixed points, it follows that $F$ is actually a smooth action on the complex manifold $X\setminus \{ p \} $. Consequently, $F$ induces a smooth manifold structure on the quotient space $(X\setminus \{ p \})/\mathbb{R} $. Since $Y$ is homeomorphic to $(X\setminus \{ p \})/\mathbb{R} $, one can see that $Y$ is homeomorphic to a compact $3$-dimensional manifold. Let $\mathcal{H}^4_d $ and $\mathcal{H}^3_{d_Y} $ denote the $4$-dimensional Hausdorff measure of $(X,d)$ and the $3$-dimensional Hausdorff measure of $(Y,d_Y)$, respectively. Since the metric measure space $(X,d,\mathcal{H}_d^4 )$ is a $RCD^* (0,4)$ space, we see that the metric measure space $(Y,d_Y,\mathcal{H}_{d_Y}^3 )$ must be a $RCD^* (2,3)$ space \cite{ck1}. It follows that $\pi_1 (Y) $ is finite, and the universal covering space $\tilde{Y}$ is a compact $3$-dimensional manifold.

By Thurston's geometrization conjecture \cite{mt1, p1, p2, p3}, we can conclude that $\tilde{Y}$ is homeomorphic to $S^3 $, and $Y$ is homeomorphic to a space form $S^3 / \Gamma $, where $\Gamma \leq SO(4) $ is a finite group of acting freely on $S^3$ by rotations. It follows that $X=C(Y)$ is homeomorphic to $\mathbb{R}^4 /\Gamma $.

Let $Z$ be the link of $X$ at $p$. Since there exists an open neighborhood $U$ of $p\in X$ which homeomorphic to $C(Z) $, one can see that $\pi_1 (Z)=\pi_1 (Z\times \mathbb{R}) \leq \pi_1 ( (\mathbb{R}^4 \setminus \{ 0 \}) /\Gamma )  $ is a finite group. Hence we can use \cite[Corollary 3.7.12]{an1} to find a finite subgroup $\Gamma_{c} \leq GL(2;\mathbb{C}) $ such that $\mathbb{C}^2 /\Gamma_c $ is homeomorphic to $(X,p)$. Since $\Gamma_c $ is finite, we can define a Hermitian metric $h(u,v) = \sum_{\gamma\in\Gamma_c} h_{0} (\gamma u ,\gamma v) $ on $\mathbb{C}^2$, where $h_0$ is the Euclidean Hermitian metric on $\mathbb{C}^2$. One can easily to see that $h(\gamma u ,\gamma v) =h(u,v) $, $\forall \gamma\in\Gamma_c $. Then we can find an $A\in GL(2;\mathbb{C} ) $ such that $A\gamma A^{-1} \in U(2) $, $\forall \gamma\in\Gamma_c $. Without loss of generality, we can assume that $\Gamma_c \leq U(2) $, and $(X,p)$ is homeomorphic to $(\mathbb{C}^2 /\Gamma_c ,0) $.
\end{proof}

Before proving the uniqueness of the topology of the tangent cones, we recall some basic concepts.

Let $\left( X ,d ,p \right) $ be the pointed Gromov-Hausdorff limit of a sequence of pointed complete K\"ahler manifolds $\left( M_i ,\omega_i ,p_i \right) $ with $\Ric \left( \omega_i \right) \geq -\Lambda \omega_i $ and $\Vol \left( B_1 \left( p_i \right) \right) \geq v $, $\forall i\in\mathbb{N} $, where $\Lambda ,v>0$ are constants. We say that a point $x \subset X$ is $ \epsilon $-$regular$ if for any tangent cone $(V,o)$ at $x$, $d_{GH} \left( B^V_r(x) ,B^{\mathbb{C}^n }_r ( 0 ) \right) < \epsilon r $. The set of $ \epsilon $-regular points is denoted by $\mathcal{R}_{\epsilon } $. Cheeger-Jiang-Naber \cite{jcwsjan1} gives the sharp Minkowski type estimate for $X\setminus \mathcal{R}_{\epsilon} $. 

If for each $x\in U $, we can find a constant $r>0 $, a sequence of points $x_i \in M_i $ and holomorphic maps $ F_i = \left( z_{i,1} , \cdots ,z_{i,N} \right) : B_{r} (x_i ) \to \mathbb{C}^N $, such that $x_i $ converges to $x$, and $F_i $ converges to an injective holomorphic map $F$ on $B_{r} (x )$, then we say that the sequence of complex structures on $M_i $ converges to the complex structure on $U $. In \cite{lggs1, lggs2}, Liu-Sz\'ekelyhidi shows that by choosing a subsequence, the complex structures of $M_i$ converge to a complex manifold structure on $\mathcal{R}_{\epsilon} $ for some $\epsilon >0$, and the complex manifold structure on $\mathcal{R}_{\epsilon} $ can be extended to a complex normal space structure on $X$ if $X$ is a metric cone or $M_i$ are polarized.

Now we are ready to prove the uniqueness of the topology of the tangent cones at $x\in X$.

\begin{thm}
\label{thmtangenttopologyuniqueness}
Let $\left( X ,d ,p \right) $ be the pointed Gromov-Hausdorff limit of a sequence of pointed complete K\"ahler surfaces $\left( M_i ,\omega_i ,p_i \right) $ with $\Ric \left( \omega_i \right) \geq -\Lambda \omega_i $ and $\Vol \left( B_1 \left( p_i \right) \right) \geq v $, $\forall i\in\mathbb{N} $, where $\Lambda ,v>0$ are constants. Fix $x\in X$. Then there exists a finite subgroup $\Gamma_x \in U(2) $ which acts freely away from the origin such that every tangent cone at $x$ is homeomorphic to $\mathbb{C}^2 /\Gamma_x $.
\end{thm}

\begin{proof}
Assume that $(V_j ,o_j) = (C(Y_j) ,o_j ) $ is a sequence of tangent cones at $x\in X$, and the sequence of metric spaces $Y_j$ converges to a metric space $Y_\infty $ in the Gromov-Hausdorff sense. Then $(V_\infty ,o_\infty ) = (C(Y_\infty ) ,o_\infty )$ is also a tangent cone at $x\in X$. Furthermore, Proposition \ref{propconeorbifold} implies that $(V_\infty ,o_\infty )$ is biholomorphic to $(\mathbb{C}^2 /\Gamma_\infty ,0) $, where $\Gamma_\infty \leq U(2) $ is a finite subgroup acts freely away from the origin of $\mathbb{C}^2 $.

By Liu-Sz\'ekelyhidi's $L^2$-estimate in \cite{lggs2}, one can find a sequence of holomorphic maps $F_j =(f_{j ,1} ,\cdots, f_{j ,N} ) : V_{j} \to \mathbb{C}^N $ converges to a proper holomorphic embedding $F_\infty =(f_{\infty ,1} ,\cdots, f_{\infty ,N} ) : V_{\infty} \to \mathbb{C}^N $ in the pointed Gromov-Hausdorff sense, where $f_{\infty ,j}$ are polynomial growth holomorphic functions on $V_\infty $. Then we can find smooth maps $ \varphi_j : B_2 (o_j ) \setminus B_{\frac{1}{4}} (o_j ) \to V_{\infty } $ for sufficiently large $j$ such that the sequence of maps gives the $C^\infty $ convergence $B_1 (o_j ) \setminus B_{\frac{1}{2}} (o_j ) \to B_1 (o_\infty ) \setminus B_{\frac{1}{2}} (o_\infty ) $ as $j\to\infty $. Fix a tangent cone $(V_x,o_x) $ at $x\in X$. Now we consider the set $ \Xi_x $ of tangent cones at $x $ containing all tangent cones at $x $ which homeomorphic to $(V_x,o_x) $. By the argument above, we see that $\Xi_x$ is open and close subset in the space of all tangent cones at $x\in X$. Since the space of all tangent cones at $x\in X$ is connected, we see that the space is $\Xi_x $, and every tangent cone at $x $ is homeomorphic to $(V_x ,o_x)$.
\end{proof}

Now we at the place to prove that the Gromov-Hausdorff limits of polarized surfaces are holomorphic orbifolds.

\begin{thm}
Let $\left( X ,d ,p \right) $ be the pointed Gromov-Hausdorff limit of a sequence of pointed complete K\"ahler surfaces $\left( M_i ,\omega_i ,p_i \right) $ with $\Ric \left( \omega_i \right) \geq -\Lambda \omega_i $ and $\Vol \left( B_1 \left( p_i \right) \right) \geq v $, $\forall i\in\mathbb{N} $, where $\Lambda ,v>0$ are constants. Assume that $\omega_i $ is polarized on $B_1 (y) $ for any $y\in M_i$, $\forall i $. Then the limit space $\left( X ,d  \right) $ is homeomorphic to a holomorphic orbifold.
\end{thm}

\begin{proof}
By \cite{lggs1}, $X$ is homeomorphic to an analytic normal surface, and the complex structure of $X$ is given by the convergence of a subsequence of the sequence of complex structures on $\{ M_i \}_{i=1}^{\infty} $. So we can find $x_1 ,\cdots ,x_N \in X $ such that $X\setminus \{ x_1,\cdots ,x_N \} $ is a complex manifold.

Write $x=x_1$. By Theorem \ref{thmtangenttopologyuniqueness}, every tangent cone at $x$ is biholomorphic to $\mathbb{C}^2 /\Gamma_x $ for some $\Gamma_x \leq U(2) $ which acts freely away from the origin. Hence for any $\epsilon >0 $, we can find a constant $\delta >0 $ such that for any $r\in (0,\delta )$, there exists a metric cone $(V_r ,o_r)$ which biholomorphic to $( \mathbb{C}^2 /\Gamma_x ,0 )$ satisfying $d_{GH} (B_r (x) ,B_r (o_r) ) <\epsilon r $. Since the metrics $\omega_i$ are polarized on unit balls, for sufficiently small $r>0$, one can apply H\"ormander's $L^2$ estimate to construct holomorphic charts on $B_{2^{10}r} (x) \setminus B_r (x) $ close to the holomorphic charts on $B_{2^{10}r} (o_r) \setminus B_r (o_r) \subset V_r $. See also \cite[Proposition 5.1]{sxz1}. Then we can find open subsets $U_j $ of $X$ for sufficiently large $j\in\mathbb{N}$ such that $B_{2^{5-4j}} (x) \setminus B_{2^{-1-4j}} (x) \subset U_{j} \subset B_{2^{6-4j}} (x) \setminus B_{2^{-3-4j}} (x) $ such that $U_j$ is diffeomorphic to $B_1 (0) \setminus \bar{B}_{2^{-4}} (0) \subset \mathbb{C}^2 /\Gamma_x $.

Then we can find a sequence of open neighborhoods $\{W_j \} $ of $x$ such that $W_j \subset B_{2^{6-4j}} (x) $, $W_{j+1} \subset W_j $, and $\bar{W}_j \setminus W_{j+1}$ is homeomorphic to $\bar{B}_1 (0) \setminus B_{2^{-1}} (0) \subset \mathbb{C}^2 /\Gamma_x $. It follows that $X$ is homeomorphic to a topological orbifold with isolated singularities. Hence we can use \cite[Corollary 3.7.12]{an1} to show that $X $ is a holomorphic orbifold.
\end{proof}

\end{document}